\documentclass [11pt,oneside]{amsart}

\usepackage{enumerate}
\usepackage{comment}
\usepackage{amssymb} \usepackage{amsfonts} \usepackage{amsmath}
\usepackage{amsthm} \usepackage{epsfig} 
\usepackage{amsmath,amssymb,enumerate,pb-diagram,pb-xy}%,pdfsync}

\usepackage[applemac]{inputenc}
\input xy
\xyoption{all}

\usepackage{caption}

\newtheorem{lemma}{Lemma}[section]
\newtheorem{thm}[lemma]{Theorem}
\newtheorem{prop}[lemma]{Proposition}

\newtheorem{prop_intro}{Proposition}
\newtheorem{cor_intro}[prop_intro]{Corollary}
\newtheorem{thm_intro}[prop_intro]{Theorem}

\newtheorem*{thm*}{Theorem}

\theoremstyle{definition}
\newtheorem{defn}[lemma]{Definition}

\newtheorem{conj_intro}[prop_intro]{Conjecture}

\newtheorem{rem}[lemma]{Remark}

\newcommand{\matN}{\ensuremath {\mathbb{N}}}
\newcommand{\R} {\ensuremath {\mathbb{R}}}
\newcommand{\Q} {\ensuremath {\mathbb{Q}}}
\newcommand{\Z} {\ensuremath {\mathbb{Z}}}

\newcommand{\matH} {\ensuremath {\mathbb{H}}}

\newcommand{\vare} {\ensuremath{\varepsilon}}
\newcommand{\vol} {{\rm Vol}}
\newcommand{\vola} {{\rm Vol}_{\rm alg}}

\newcommand{\wdtM}{\widetilde{M}}

\newcommand{\calM} {\ensuremath {\mathcal{M}}}

\newcommand{\mathno}{\ensuremath{\overline{\matH^n}}}

\newcommand{\str} {\ensuremath {{\rm str}}}

\newcommand {\bb} {\partial}

\author{Michelle Bucher}
\author{Roberto Frigerio}
\author{Cristina Pagliantini}

\address{Section de Math\'ematiques, 2-4 rue du Li\`evre, Case postale 64, 1211 Gen\`eve 4, Suisse}
\email{michelle.bucher-karlsson@unige.ch}

\address{Dipartimento di Matematica \\
Universit\`a di Pisa \\
Largo B.~Pontecorvo 5 \\
56127 Pisa, Italy}

\email{frigerio@dm.unipi.it}

\address{Department Mathematik, ETH Zentrum, R\"amistrasse 101, 8092 Z\"urich}
\email{cristina.pagliantini@math.ethz.ch}

\thanks{Michelle Bucher was supported by Swiss National Science Foundation 
project PP00P2-128309/1.  The authors thank the Institute Mittag-Leffler in Djursholm, Sweden, 
for their warm hospitality during  the preparation of this paper.}

\title[A quantitative version of a theorem by Jungreis]{A quantitative version of a theorem by Jungreis}

\subjclass[2000]{}


\thanks{}

\begin{document}

\begin{abstract}
A fundamental result by Gromov and Thurston asserts that, if $M$ is a closed hyperbolic $n$-manifold, then
the simplicial volume $\|M\|$ of $M$ is equal to $\vol (M)/v_n$, where $v_n$ is a constant depending only on the
dimension of $M$. The same result also holds for complete finite-volume hyperbolic manifolds without boundary,
while Jungreis proved that the ratio $\vol (M)/\|M\|$ is strictly smaller than $v_n$ if $M$ is compact with non-empty geodesic boundary.
We prove here a quantitative version of Jungreis' result for $n\geq 4$, which bounds from below the ratio $\|M\|/\vol (M)$
in terms of the ratio $\vol (\partial M)/\vol(M)$. As a consequence, we show that a sequence $\{M_i\}$ of compact hyperbolic $n$-manifolds
with geodesic boundary satisfies $\lim_i \vol(M_i)/\|M_i\|=v_n$ if and only if $\lim_i \vol(\partial M_i)/\vol(M_i)=0$.

We also provide estimates of the simplicial volume of hyperbolic manifolds with geodesic boundary in dimension three. 
\end{abstract}
\maketitle

\section*{Introduction}
The simplicial volume is a homotopy invariant of compact manifolds introduced by Gromov in his pioneering work~\cite{Gromov}.
If $M$ is a connected, compact, oriented manifold with
(possibly empty) boundary, then the simplicial volume of $M$, denoted by $\|M, \bb M\|$, is the 
infimum of the sum of the absolute values of the coefficients over all singular chains representing 
the real fundamental cycle of $M$ (see Section~\ref{preliminary}). If $\bb M=\emptyset$ we denote the simplicial volume of $M$ simply by $\|M\|$.
If $M$ is open, the fundamental class and the simplicial 
volume of $M$ admit analogous definitions in the context of homology of \emph{locally finite} chains, but in this paper we will restrict our attention  to compact manifolds:
unless otherwise stated, henceforth every manifold is assumed to be compact. Observe that the simplicial volume of an oriented manifold does not depend on its orientation and that it is straightforward
 to extend the definition also to  nonorientable or disconnected manifolds:
if $M$ is connected and nonorientable, then its simplicial volume is equal
to one half of the simplicial volume of its orientable double covering, and the simplicial
volume of any  manifold is the sum of the simplicial volumes of its connected components.

Several vanishing and nonvanishing results
for the simplicial volume
are available by now, 
but the exact value of nonvanishing simplicial volumes is known only in a very few cases. 
A celebrated result by Gromov and 
Thurston~\cite{Gromov,Thurston} implies that, if $M$ is a hyperbolic $n$-manifold 
without boundary $M$, then
\begin{equation}\label{closedhyp}
\| M\|=\frac{\vol(M)}{v_n}\ ,
\end{equation}
where $\vol(M)$ is  the Riemannian volume of $M$
and $v_n$ is the volume of the regular ideal geodesic $n$-simplex in hyperbolic space.
In the closed case, the only other exact computation of nonvanishing simplicial volume is for the product of two closed hyperbolic surfaces
or more generally manifolds locally isometric to the product of two hyperbolic planes \cite{Bucher3}. 

After replacing $\|M\|$ with $\|M,\partial M\|$,
equality~\eqref{closedhyp}
holds also when $M$ is the natural compactification of any complete noncompact hyperbolic $n$-manifold of finite volume
without boundary (see \emph{e.g.}~\cite{stefano,FP,FM,BB}). In this case, every component of $\partial M$ supports a Euclidean structure,
so $\|\partial M\|=0$. In the case when $\|\partial M\|\neq 0$, 
the only exact computations of the simplicial volume of compact manifolds with nonempty boundary are provided in \cite{BFP15} 
for products of surfaces with the interval and for compact $3$-manifolds obtained by adding $1$-handles to Seifert manifolds.
Building on these examples, more values for the simplicial volume can be obtained by surgery or by taking 
connected sums or amalgamated sums over submanifolds with amenable fundamental group (see \emph{e.g.}~\cite{Gromov, Kuessner,  BBFIPP}).
However, the exact value of the simplicial volume is not known for any compact $n$-manifold $M$ such that $\|\partial M\|\neq 0$, $n\geq 4$, and
for any hyperbolic $n$-manifold with non-empty geodesic boundary, $n\geq 3$.

\subsection*{Hyperbolic manifolds with geodesic boundary}
%The simplicial volume of hyperbolic manifolds has been widely studied. 
Jungreis proved in~\cite{Jung} that, if $M$ is an $n$-dimensional 
hyperbolic manifold
with nonempty geodesic boundary, $n\geq 3$, then 
%We have already mentioned that, if $M$ is a closed
%hyperbolic $n$-manifold, then 
%$$
%\|M\| =\frac{\vol(M)}{v_n}\ ,
%$$
%where $v_n$ denotes the volume of the regular ideal simplex in hyperbolic $n$-space,
%while in the 
%case of hyperbolic manifolds with nonempty geodesic boundary, $n\geq 3$, Jungreis proved~\cite{Jung}  that
%this equality 
%does not hold by showing that
\begin{equation}\label{jung:eq}
\|M,\bb M\| > \frac{\vol(M)}{v_n} \ .
\end{equation}
In Section~\ref{proof-main:sec}  we provide a quantitative version of Jungreis' result
in the case when $n\geq 4$. More precisely, we prove the following:

\begin{thm_intro}\label{hyp:high:thm}
Let $n\geq 4$. Then there exists a constant $\eta_n>0$ depending only on $n$ such that
$$
%\frac{\vol(M)}{\|M,\bb M\|}\leq v_n-\eta_n\cdot \frac{\vol(\bb M)}{\vol(M)}\ .
\frac{\|M,\bb M\|}{\vol(M)}\geq \frac{1}{v_n}+\eta_n\cdot \frac{\vol(\bb M)}{\vol(M)}\ .
$$
\end{thm_intro}

It is well-known that $\|M,\bb M\|=\vol(M)/v_2=\vol(M)/\pi$
for every hyperbolic surface with geodesic boundary $M$, so Theorem~\ref{hyp:high:thm}
cannot be true in dimension $2$. The $3$-dimensional case is still open.

Theorem~\ref{hyp:high:thm} states that, if $n\geq 4$, then $\vol(M)/\|M,\bb M\|$ cannot
approach $v_n$ unless the 
$(n-1)$-dimensional volume
of $\bb M$ is small with respect to the volume of $M$. 
On the other hand, it is known that 
$\vol(M)/\|M,\bb M\|$ indeed approaches $v_n$ if 
$\vol(\bb M)/\vol(M)$ is small.
%the 
%$(n-1)$-dimensional volume
%of $\bb M$ is small with respect to the volume of $M$. 
In fact, the following result is proved in~\cite{FP} for $n\geq 3$:
for every $\vare>0$ there exists $\delta>0$ such that 
$$
\frac{\vol(\bb M)}{\vol(M)}<\delta\quad\Longrightarrow \quad\frac{\vol(M)}{\|M,\bb M\|}\geq v_n-\vare
$$
for every 
hyperbolic $n$-manifold $M$ with nonempty geodesic
boundary.

Note that in particular, the ratio between $\| M,\bb M\|$ and $\vol(M)$ does not depend 
only on the dimension of $M$. Putting together this result with Theorem~\ref{hyp:high:thm},
we obtain, for $n\geq 4$, a complete characterization of hyperbolic $n$-manifolds 
with geodesic boundary whose simplicial volume is close to the bound given by inequality~\eqref{jung:eq}:

\begin{cor_intro}\label{hyp:high:cor}
Let $n\geq 4$, and let $M_i$ be a sequence of  hyperbolic $n$-manifolds with nonempty geodesic
boundary. Then 
$$
\lim_{i\to \infty} \frac{\vol(M_i)}{\|M_i,\bb M_i\|}=v_n
\quad \Longleftrightarrow \quad
\lim_{i\to \infty} \frac{\vol(\bb M_i)}{\vol(M_i)}=0\ .
$$
\end{cor_intro}

\subsection*{The $3$-dimensional case}
Let $M$ be a boundary irreducible aspherical $3$-manifold. In~\cite[Theorem 1.4]{BFP15}
we proved the following sharp lower bound for the simplicial volume of $M$ in terms of the simplicial volume of $\partial M$:
\begin{equation}\label{54eq}
 \|M,\partial M\|\geq \frac{5}{4} \|\partial M\|\ .
\end{equation}

Every  hyperbolic manifold with geodesic boundary
is aspherical and boundary irreducible.
Therefore,
even if Theorem~\ref{hyp:high:thm} is still open in dimension $3$,  inequality \eqref{54eq} %Theorem 1.4 in \cite{BFP15}
may be exploited to show that, if $\vol(\bb M)$ is big with respect to $\vol(M)$, then
indeed the simplicial volume of $M$ is bounded away from ${\vol(M)}/v_3$. In the same spirit, by combining combinatorial
and geometric arguments,
in Section~\ref{proof-hyp:sec} we prove the following result, where 
$$
G=\frac{1}{1^2}-\frac{1}{3^2}+\frac{1}{5^2}-\frac{1}{7^2}+\ldots\,  \approx 0.915965
$$ 
is Catalan's constant.

\begin{thm_intro}\label{hyperbolic:thm}
 Let $M$ be a hyperbolic $3$-manifold with nonempty geodesic boundary.
Then
$$
%\|M,\bb M\| \geq \frac{7\|\bb M\| (v_3-G)+2\vol(M)}{2(3v_3-2G)} 
\|M,\bb M\| \geq \frac{\vol(M)}{v_3}+\frac{v_3-G}{2(3v_3-2G)}\left(7\|\bb M\|-4\frac{\vol(M)}{v_3}\right)\ .
%\frac{7\|\bb M\| (v_3-G)+2\vol(M)}{2(3v_3-2G)} 
%\ \left(\, > \frac{\vol(M)}{v_3}\right)\ .
$$
\end{thm_intro}

At the moment, no precise computation of the simplicial volume of 
 hyperbolic $3$-manifold with nonempty geodesic boundary is known. 
Let us briefly introduce some families of examples for which
the bound
provided by Theorem~\ref{hyperbolic:thm} is sharper than both 
Jungreis' bound~\eqref{jung:eq} and bound~\eqref{54eq}.

For every $g\geq 2$ let $\overline\calM_g$ be the set of hyperbolic $3$-manifolds $M$
with connected geodesic boundary such that $\chi(\bb M)=2-2g$ (so $\bb M$, if orientable, is the closed
orientable surface of genus $g$).
Recall that for every $3$-manifold with boundary $M$
the equality $\chi(\bb M)=2\chi (M)$ holds, and in particular $\chi(\bb M)$ is even.
%Since the Euler characteristic of the boundary of any compact $3$-manifold
%is even, 
Therefore, the union $\bigcup_{g\geq 2} \overline\calM_g$ coincides with the set
of hyperbolic $3$-manifolds with connected geodesic boundary.

For every $g\geq 2$ we denote by $\calM_g$ the set of 
3-manifolds with boundary $M$ that admit an ideal triangulation by $g$ tetrahedra
and have Euler characteristic $\chi(M)=1-g$
%is equal to a closed surface of genus $g$ 
(see Section~\ref{last:sec}
for the definition of ideal triangulation).
%It is well-known that $\calM_g$
%is nonempty for every $g\geq 2$ (see \emph{e.g.}~\cite{KM,FriMaPe}).
%Moreover, 
Every element of $\calM_g$ has connected boundary 
and supports
a hyperbolic structure with geodesic boundary (which is unique
by Mostow rigidity), hence $\calM_g\subseteq \overline\calM_g$
(see Proposition~\ref{Mg:prop}).
Furthermore, Miyamoto proved in~\cite{KM}
that elements of $\calM_g$ are exactly the ones having the smallest volume
among the elements of $\overline\calM_g$. In particular, $\calM_g$ is nonempty
for every $g\geq 2$.
The eight elements of $\calM_2$ are exactly the smallest hyperbolic manifolds
with nonempty geodesic boundary~\cite{Kojima,KM}. 

As a consequence of Corollary~\ref{hyp:high:cor},
the simplicial volume and the Riemannian volume of hyperbolic $3$-manifolds with nonempty geodesic boundary are 
not related by a universal proportionality constant. Nevertheless, it is reasonable to expect
that these invariants are closely related to each other. Therefore,
we make here the following conjecture:
\begin{conj_intro}\label{minimal:hyp:conj}
For $g\geq 2$, the elements of $\calM_g$ are exactly the ones having the smallest
simplicial volume among the elements of $\overline\calM_g$.
Moreover, the eight elements of $\calM_2$ are the hyperbolic manifolds
with nonempty geodesic boundary having the smallest simplicial volume.
\end{conj_intro}

Together with Miyamoto's results about volumes of hyperbolic
manifolds with geodesic boundary~\cite{KM}, Theorem~\ref{hyperbolic:thm}
implies the following (see Section~\ref{last:sec}):

\begin{cor_intro}\label{minimal:cor}
If $M\in \overline\calM_2$, 
then $\|M,\bb M\|\geq 6.461\approx 1.615 \cdot \|\bb M\|$.
If $M\in\overline\calM_3$, then
$\|M,\bb M\|\geq 10.882\approx 1.360 \cdot \|\bb M\|$.
If $M\in\overline\calM_4$, then
$\|M,\bb M\|\geq 15.165\approx 1.264 \cdot \|\bb M\|$.
\end{cor_intro}

As we will see in Section~\ref{last:sec}, the corollary shows that Theorem~\ref{hyperbolic:thm}
indeed improves Jungreis' inequality~\eqref{jung:eq}
and bound~\eqref{54eq} in some cases. More precisely we will show that if $M\in \calM_2\cup\calM_3\cup\calM_4$, the bounds provided by Theorem~\ref{hyperbolic:thm} and Corollary~\ref{minimal:cor} coincide, and
are sharper than the bounds provided by inequalities~\eqref{jung:eq} and \eqref{54eq}, while
if $M\in\calM_g$, $g\geq 5$, then the sharpest known bound for $\|M,\bb M\|$ is provided by inequality~\eqref{54eq}.

\section{Simplicial volume}\label{preliminary}
Let $X$ be a topological space and $Y\subseteq X$ a (possibly empty) subspace of $X$. Let
$R$ be a normed ring. Henceforth we confine ourself to $R=\R,\Q$ or $\Z$, where each of these rings is endowed with the norm given by the absolute value.
For $i\in\matN$ we denote by $C_i(X;R)$ the module
of singular $i$-chains over $R$, \emph{i.e.}~the $R$-module freely generated by the set $S_i(X)$ of singular $i$-simplices
in $X$, and we set as usual $C_i(X,Y;R)=C_i(X;R)/C_i(Y;R)$.
%We observe that the $R$-module $C_i (X,Y;R)$ is free and admits the preferred
%basis given by the classes of the singular simplices whose image is not contained in $Y$.
Notice that we will  identify $C_i(X,Y;R)$ with the free $R$-module generated
by $S_i(X)\setminus S_i(Y)$. In particular, for $z\in C_i(X,Y;R)$, it will be understood from the
equality 
$z=\sum_{k=1}^n a_k\sigma_k$ that $\sigma_k\neq\sigma_h$ for $k\neq h$, and $\sigma_k\notin S_i(Y)$ for every $k$. %, while we could also write $z=\sum_{k=1}^n a_k [\sigma_k]$ in $C_i(X;R)/C_i(Y;R)$. 
We denote by $H_\ast (X,Y;R)$ the singular homology of the pair $(X,Y)$ with coefficients
in $R$, \emph{i.e.}~the homology of the complex $(C_{\ast} (X,Y;R),d_\ast)$,
where $d_\ast$ is the usual differential.

We endow the $R$-module $C_i (X,Y;R)$ with the $L^1$-norm defined by
$$
\left\| \sum_{\sigma} a_\sigma\sigma \right\|_R=\sum_{\sigma} |a_\sigma|\ ,
$$
where $\sigma$ ranges over the simplices in 
$S_i (X)\setminus S_i(Y)$. We denote simply by $\|\cdot \|$ the norm $\|\cdot\|_\R$.
%By taking the infimum over representatives, 
The norm $\|\cdot \|_R$
descends to a seminorm on $H_\ast (X,Y;R)$, which is still denoted by $\|\cdot \|_R$ and
is defined as follows:
if $\alpha\in H_i (X,Y;R)$, then
$$
\|\alpha \|_R  =  \inf \{\|\beta \|_R,\, \beta\in C_i (X,Y;R),\, d\beta=0,\, [\beta ]=\alpha  \} \ .
$$
%Note that although $\|\cdot\|_\Z$ is often called a seminorm in the literature, it is technically not so as it is not multiplicative in general. 
The \emph{real} singular homology module $H_\ast(X,Y;\R)$ 
and the seminorm $\|\cdot\|_\R$ will be simply denoted
by $H_\ast(X,Y)$ and $\|\cdot\|$ respectively.

%\subsection*{Simplicial volume}
If $M$ is a connected oriented  $n$-manifold with (possibly empty) boundary $\bb M$,
then we denote by $[M,\bb M]_R$ the fundamental class of the pair $(M,\partial M)$ with coefficients
in $R$. 
The following definition is due to Gromov~\cite{Gromov, Thurston}:

\begin{defn}
The \emph{simplicial volume} of $M$ is 
$$\|M,\bb M\|= \|[M,\bb M]_\R\|=\| [M,\bb M]_\R\|_\R\ .$$ 
The \emph{rational}, respectively \emph{integral}, simplicial volume of $M$
is defined as $\|M,\bb M\|_\Q=\|[M,\bb M]_\Q\|_\Q$, respectively $\|M,\bb M\|_\Z=\|[M,\bb M]_\Z\|_\Z$.
\end{defn}

Just as in the real case, the  rational and the  integral
simplicial volume may be defined also when $M$ is disconnected or nonorientable.
Of course we have the inequalities $\| M,\bb M\|\leq \|M,\bb M\|_\Q\leq \|M,\bb M\|_\Z$.
Using that $\Q$ is dense in $\R$, it may be shown  
in fact that $\|M,\bb M\|=\|M,\bb M\|_\Q$
and  we provide here a complete proof of this fact.
%\begin{comment}
\begin{prop}\label{rational:prop}
For every $n$-manifold $M$, the real and rational simplicial volumes are equal,
$$
\| M,\bb M\|=\| M,\bb M\|_\Q\ .
$$
\end{prop}
\begin{proof}
We have to show that $\|M,\bb M\|_\Q\leq \|M, \bb M\|$. Let $\varepsilon>0$ be fixed, and let $z=\sum_{i=1}^k a_i\sigma_i$
be a real fundamental cycle for $M$ such that $\| z\|=\sum_{i=1}^k |a_i|\leq \|M,\bb M\|+\varepsilon$.
We set
$$
H_\R=\left\{(x_1,\ldots,x_k)\in\R^k\, \big|\, \sum_{i=1}^k x_i\sigma_i\ {\rm is\ a\ relative\ cycle}\right\} \subseteq \R^k\ .
$$
Of course, $H_\R$ is a linear subspace of $\R^k$. Since $H_\R$ is defined by a system of equations
with integral coefficients, if $H_\Q=H_\R\cap \Q^k$, then $H_\Q$ is dense in $H_\R$. As a consequence, we may find sequences
of rational coefficients
$\{\alpha_i^j\}_{j\in \mathbb{N}}\subseteq \Q$, $i=1,\ldots,k$ such that 
$z^j=\sum_{i=1}^k \alpha_i^j \sigma_i$ is a rational cycle for every $j\in\mathbb{N}$, and
$\lim_j \alpha_i^j=a_i$ for every $i=1,\ldots,k$. This implies in particular that
$\lim_j \|z^j\|_\Q=\| z\|$, so we are left to show that the $z^j$'s may be chosen among the representatives
of the rational fundamental class of $M$.

%To this aim, we consider the real fundamental coclass of $M$, \emph{i.e.}~the unique cohomology class
%$[M,\bb M]^\R\in H^n(M,\partial M;\R)$ such that $$\langle [M,\bb M]^\R,[M,\bb M]_\R\rangle=1\ ,$$ where $\langle \cdot,\cdot\rangle$
%is the usual Kronecker product between homology and cohomology. Since $[M,\bb M]^\R$
%is the image of an integral class via  the change of coefficients homomorphism
%$H^n(M,\partial M;\Z)\to H^n(M,\partial M;\R)$, as a representative for $[M,\bb M]^\R$ we may choose a cocycle
%$\varphi\colon C_n(M,\partial M;\R)\to \R$ such that $\varphi (\sigma_i)\in \Z$ for every $i=1,\ldots,k$.
Let $\lambda_j\in \mathbb{Q}$ be defined by $[z^j]=\lambda_j\cdot[M,\bb M]$
(such a $\lambda_j$ exists because $[M,\bb M]$ lies in the image of $H_n(M,\bb M;\mathbb{Q})$
in $H_n(M,\bb M;\mathbb{R})$ under the change of coefficients homomorphism). 
The Universal Coefficient Theorem provides a real cocyle $\varphi\colon C_n(M,\bb M;\mathbb{R})\to \mathbb{R}$
such that $\varphi(z)=1$. Observe that $\varphi(z^j)=\lambda_j$, so from
$\lim_j \alpha_i^j=a_i$ we deduce that $\lim_j \lambda_j=\lim_j \varphi(z^j)=\varphi(z)=1$.
For large $j$ we may thus define
$w^j=\lambda_j^{-1}\cdot z^j\in C_n(M,\partial M;\Q)$, and by construction $w^j$ represents the 
rational fundamental class of $M$. Finally, we have
$$
\lim_j \| w^j\|_\Q=\lim_j \frac{\|z^j\|_\Q}{\lambda_j}=\| z\|\leq \|M,\bb M \|+\vare\ ,
 $$
which finishes the proof of the proposition.
%Let $\lambda_j=\varphi(z^j)$. Since $z^j$ has rational coefficients we have $\lambda_j\in\Q$. Since $\lim_j \alpha_i^j=a_i$ we have $\lim_j \lambda_j=\varphi(z)=1$. For large $j$ we may thus define
%$w^j=\lambda_j^{-1}\cdot z^j\in C_n(M,\partial M;\Q)$, and by construction $w^j$ represents the 
%rational fundamental class of $M$. Finally, we have
%$$
%\lim_j \| w^j\|_\Q=\lim_j \frac{\|z^j\|_\Q}{\lambda_j}=\| z\|\leq \|M,\bb M \|+\vare\ ,
 %$$
%which finishes the proof of the proposition.
\end{proof}
%\end{comment}

On the contrary, the  inequality $\|M,\bb M\|\leq \|M,\bb M\|_\Z$
is not an equality in general, for instance $\|S^1\|=0$ but $\|S^1\|_\Z\geq 1$.
The integral simplicial volume 
does not behave as nicely as the rational or real simplicial volume.
For example,
it follows from the definition that $\|M \|_\Z\geq 1$ for every manifold $M$. Therefore, the integral simplicial volume cannot be multiplicative with respect to finite coverings (otherwise
it should vanish on manifolds that admit finite nontrivial self-coverings,
as $S^1$). %Another defect is that the $L^1$-seminorm on integral homology is not really a seminorm, since the equality
%$\|n\cdot \alpha\|_\Z=|n|\cdot \|\alpha\|_\Z$, for $n\in \Z$, $\alpha\in H_\ast(X,Y;\Z)$, may not hold. Indeed, it is easy to see that $\| n\cdot [S^1]\|_\Z=1$ for every $n\in\Z\setminus\{0\}$.
Nevertheless, we will use integral cycles extensively,
as they admit a clear geometric interpretation in terms of pseudomanifolds (see~Section~\ref{main:sec}). 
In order to follow this strategy, we need the following obvious
consequence of the equality $\|M,\partial M\|_\Q=\|M,\partial M\|$.

\begin{lemma}\label{integral:estimate:lemma}
Let $M$ be connected and oriented, and let $\varepsilon>0$ be given. Then, there exists an integral cycle
$z\in C_n(M,\partial M;\Z)$ such that 
$$
\frac{\|z\|_\Z}{d}\leq \|M,\bb M\|+\vare\ ,
$$
where $[z]=d\cdot [M,\bb M]_\Z$ and $d> 0$ is an integer.
\end{lemma}
%\begin{proof}
%Since $\| M,\bb M\|=\| M,\bb M\|_\Q$, a rational %cycle $z'\in\ C_n(M,\partial M;\Q)$
%exists such that $[z']_\Q=[M,\bb M]_\Q$ and $\| %z'\|_\Q\leq \|M,\bb M\|+\vare$.
%Of course there exists $d\in\matN\setminus\{0\}$ %such that $z=d\cdot z'$ 
%lies in $ C_n(M,\partial M;\Z)$. The integral %cycle $z$ satisfies the desired properties.
%\end{proof}

Moreover, the boundary of a fundamental cycle for $M$ is equal to the sum of one fundamental cycle for each component
of $\partial M$, so we also have:
\begin{equation}\label{eq:boundary}
\|\partial M\|\leq \frac{\|\partial z\|_\Z}{d}\ .
\end{equation}

\begin{rem} The statements and the proofs of Proposition \ref{rational:prop} and Lemma \ref{integral:estimate:lemma} hold more generally after replacing the fundamental class $[M,\partial M]_{\mathbb{Q}}$ by any rational homology class.
In other words, for every $i\in\mathbb{N}$ the change of coefficients map $H_i(M,\bb M;\mathbb{Q})\to H_i(M,\bb M;\mathbb{R})$
is norm-preserving.
\end{rem} 
 
%\subsection{Elementary properties of the simplicial volume}
Finally, let us list some elementary properties of the simplicial volume which will be needed later. 

\begin{prop}[\cite{Gromov}]\label{degree:prop}
 Let $M,N$ be connected oriented manifolds of the same dimension, and suppose that
either $M,N$ are both closed, or they both have nonempty boundary. Let $f\colon N\to M$
be a map of degree $d$. Then
$$
\|N,\bb N\|\geq |d|\cdot \| M,\bb M\| \ .
$$ 
%If $f$ is a covering, then the equality $\|N\|= d\cdot \| M\|$ holds.
\end{prop}

The following well-known result describes the simplicial volume of 
closed surfaces. In fact, the same statement also holds for connected surfaces with boundary.
 
\begin{prop}[\cite{Gromov}]\label{vol:surf:prop}
Let $S$ be a closed surface. Then
$$
\| S\|=\max \{0,-2\chi (S)\}\ .
$$
\end{prop}

\section{Representing cycles via pseudomanifolds}\label{main:sec}
This section is devoted to recall the notion of pseudomanifold and the interpretation of integral cycles 
by means of pseudomanifolds. 

\subsection*{Pseudomanifolds}
Let $n\in\matN$. An \emph{$n$-dimensional pseudomanifold} $P$ consists of a finite number of copies
of the standard $n$-simplex, a choice of pairs of $(n-1)$-dimensional faces of $n$-simplices such that each face
appears in at most one of these pairs, and an affine
identification between the faces of each pair. We allow pairs of distinct faces in the same $n$-simplex. 
It is \emph{orientable} if orientations on the simplices of $P$
may be chosen in such a way that the affine identifications between the paired faces
(endowed with the induced orientations)
are all orientation-reversing.
A face which does not belong to any 
pair of identified faces is a \emph{boundary} face.

We denote 
by $|P|$ the \emph{topological realization} of $P$, \emph{i.e.}~the quotient space of the union of the simplices
by the equivalence relation generated by the identification maps.
We say that
$P$ is connected if $|P|$ is.
We denote by $\partial |P|$
the image in $|P|$ of the boundary faces of $P$, and 
we say that $P$ is \emph{without boundary} if $\partial |P|=\emptyset$.

A \emph{$k$-dimensional face}
of $|P|$ is the image in $|P|$ of a $k$-dimensional face of a simplex of $P$.
Usually, we refer to $1$-dimensional, respectively $0$-dimensional
faces of $P$ and  $|P|$ as to \emph{edges}, respectively \emph{vertices}
of $P$ and $|P|$.

%Observe that  we do not require 
%the topological realization of a pseudomanifold to be connected.
%In this way, the boundary of a pseudomanifold is itself a pseudomanifold (see below). 

It is well-known that, 
if $P$ is an $n$-dimensional pseudomanifold, $n\geq 3$, then $|P|$ does not need to be a manifold.
However, 
in the $3$-dimensional orientable case, singularities may occur only at vertices
(and it is not difficult to construct examples where they indeed occur).
Let us be more precise, and state 
the following well-known result
(see \emph{e.g.}~\cite[pages 108-109]{Hat}):
%However, the topological realization of any $2$-pseudomanifold is a genuine surface, and
%singularities of orientable $3$-dimensional pseudomanifolds can occur only at vertices
%(see \emph{e.g.}~\cite[pages 108-109]{Hat} and ~\cite[Exercise 1.3.2(b)]{Thurston2}).

%\begin{comment}
\begin{lemma}\label{basic:gluing:lemma}
 Let $P$ be an orientable $n$-dimensional pseudomanifold, 
and let $V_k\subseteq |P|$ be the union
of the $k$-dimensional faces of $|P|$. 
Then $|P|\setminus V_{n-3}$ is an orientable
manifold. In particular, if $P$ is an orientable $2$-dimensional pseudomanifold without boundary,
then $|P|$ is an orientable surface without boundary.
\end{lemma}

Moreover, the boundary of the topological realization of an $n$-dimensional pseudomanifold $P$ is naturally the topological realization
of an (orientable) $(n-1)$-pseudomanifold $\partial P$ without boundary, and an orientation of $P$ canonically induces an
orientation on $\partial P$. In particular, if $P$ is orientable and $3$-dimensional, then $\partial |P|$ is a finite union of orientable
closed surfaces.

\subsection*{The pseudomanifold associated to an integral cycle}\label{gluing:cycle:sub}
Let $M$ be an oriented connected $n$-dimensional manifold with (possibly empty) boundary $\partial M$.
It is well-known that every integral relative cycle on $(M,\partial M)$ can be represented by a map from a suitable pseudomanifold
to $M$. Let us describe this procedure in detail in the case we are interested in, \emph{i.e.}~in the case of $n$-dimensional integral cycles
(see also~\cite[pages 108-109]{Hat}).

Let $z=\sum_{i=1}^k \vare_i \sigma_i$ be an $n$-dimensional relative cycle in
$C_n(M,\bb M;\mathbb{Z})$,
where 
$\sigma_i$ is a singular $n$-simplex on $M$, and
$\vare_i=\pm 1$ for every $i$ (note that here we do not assume that $\sigma_i\neq \sigma_j$ for  $i\neq j$).
We construct an $n$-pseudomanifold associated to $z$ as follows. Let us consider $k$ distinct copies
$\Delta^n_1,\ldots,\Delta^n_k$ of the standard $n$-simplex $\Delta^n$.
For every $i$ we fix an identification between $\Delta^n_i$ and $\Delta^n$, so that we may consider
$\sigma_i$ as defined on $\Delta^n_i$.
For every $i=1,\ldots,k$, $j=0,\ldots, n$, we denote by $F^i_j$ the $j$-th face of $\Delta^n_i$,
and by $\partial^i_j\colon \Delta^{n-1}\to F^i_j\subseteq \Delta_i^n$ the usual face inclusion.
We say that the distinct faces $F^i_j$ and $F^{i'}_{j'}$ form a \emph{canceling pair} if 
$\sigma_i|_{F^i_j}=\sigma_{i'}|_{F^{i'}_{j'}}$ and $(-1)^j\vare_i+(-1)^{j'}\vare_{i'}=0$. This
is equivalent to say that, when computing the boundary $\partial z$ of $z$, the pair of $(n-1)$-simplices
arising from the restrictions of $\sigma_i$ and $\sigma_{i'}$ 
to $F_j^i$ and $F_{j'}^{i'}$ cancel each other. 

Let us define a pseudomanifold $P$ as follows.
The simplices of $P$ are $\Delta^n_1,\ldots,\Delta^n_k$, and we identify the faces
belonging to a maximal collection of canceling pairs
(note that such a family is not uniquely determined). If $F_i^j$, $F_{i'}^{j'}$ are paired faces, we identify them via the affine diffeomorphism
$\partial_{i'}^{j'}\circ (\partial_i^j)^{-1}\colon F_i^j\to F_{i'}^{j'}$.
We observe that $P$ is orientable: in fact,
we can define an orientation on $P$ by endowing $\Delta^n_i$ with the standard orientation
of $\Delta^n$ if $\vare_i=1$, and with the reverse orientation if $\vare_i=-1$.

By construction, the maps $\sigma_1,\ldots,\sigma_k$ glue up to a well-defined continuous map
$f\colon |P|\to M$. For every $i=1,\ldots, k$, let  $\hat{\sigma}_i\colon \Delta^n\to |P|$ be the singular simplex obtained by composing
the identification $\Delta^n\cong \Delta^n_i$ with the quotient map with values in $|P|$,
and let us set $z_P=\sum_{i=1}^k \vare_i \hat{\sigma}_i$.
Then
the chain $z_P$ is a relative cycle in $C_n(|P|,\partial |P|;\Z)$ and the map 
$f_\ast$ induced by $f\colon (|P|,\bb |P|)\to (M,\bb M)$ on integral singular chains sends $z_P$ to $f_\ast(z_P)=z$.

By Lemma~\ref{integral:estimate:lemma}, the simplicial volume 
of a connected oriented $n$-manifold can be computed from \emph{integral} cycles.
By exploiting Thurston's straightening procedure for simplices, 
in Proposition \ref{cyclehyp:prop} we will show 
that such efficient cycles
may be represented by $n$-pseudomanifolds with additional useful properties.

\section{Geometric properties of straight cycles}\label{hyperbolic:sec}
The \emph{straightening procedure} for simplices was introduced by Thurston in~\cite{Thurston},
in order to bound from below the simplicial volume of hyperbolic manifolds. 
%In the context of hyperbolic manifolds, the straightening procedure introduced
%in Section~\ref{aspherical:sec} admits a useful geometric description, which
%dates back to Thurston~\cite{Thurston}: 
The universal covering of a hyperbolic $n$-manifold with geodesic boundary
is a convex subset of the hyperbolic 
space $\matH^n$, and the support of any straight simplex
is just the image of a geodesic simplex of $\matH^n$ via the universal covering
projection. 
As a consequence, to compute the simplicial volume of a hyperbolic manifold with geodesic boundary we may restrict to considering only cycles supported by (projections of) geodesic simplices. 

%Since the 
%(weighted)
%sum of the volumes of the simplices of a real fundamental cycle of $M$
%has to be equal to
%the volume of $M$, one can obtain a lower bound on the simplicial
%volume from an upper bound on the volume of geodesic simplices. 
%In order
%to exploit this fact in our cases of interest, we begin by 
%recalling some results about hyperbolic geodesic simplices.

\subsection*{Geodesic simplices}
Let $\mathno=\matH^n\cup \bb\matH^n$ be the usual compactification of the hyperbolic space $\matH^n$.
We recall that every pair of points of $\mathno$ is connected by a unique geodesic segment (which 
has infinite length if any of its endpoints lies in $\partial\matH^n$). A subset in $\mathno$ is \emph{convex}
 if whenever it contains a pair of points it also contains the geodesic segment connecting them. The \emph{convex hull} of a set $A$ is defined as usual as the intersection of all convex sets containing $A$.
A \emph{(geodesic) $k$-simplex} $\Delta$ in $\mathno$ is the convex hull of $k+1$ points in $\mathno$, called \emph{vertices}. We say that a $k$-simplex is:
\begin{itemize}
\item \emph{ideal} if all its vertices lie in $\partial\matH^n$,
\item \emph{regular} if every permutation of its vertices is induced by an isometry of $\matH^n$,
\item \emph{degenerate} if it is contained in a $(k-1)$-dimensional subspace of $\mathno$.
\end{itemize}

As above, we denote by $v_n$ the volume of the regular ideal simplex in $\mathno$. The following result
characterizes hyperbolic geodesic simplices of maximal volume, and plays a fundamental
role in the study of the simplicial volume of hyperbolic manifolds:

\begin{thm}[\cite{HM, Pe}] \label{maximal:teo}
Let $\Delta$ be an $n$-simplex in $\mathno$. Then $\vol(\Delta)\leqslant v_n$, with equality if and only if $\Delta$ is ideal and regular.
\end{thm}

Let $\Delta$ be a nondegenerate geodesic $n$-simplex, and let $E$ be an $(n-2)$-dimensional face of $\Delta$. The \emph{dihedral angle}
$\alpha (\Delta,E)$ 
of $\Delta$ at $E$ is defined as follows: let $p$ be a point in $E\cap \matH^n$, and let $H\subseteq \matH^n$ 
be the unique $2$-dimensional geodesic plane which intersects $E$ orthogonally in $p$. We set $\alpha(\Delta,E)$ to be equal to the angle in $p$ of the polygon
$\Delta\cap H$ of $H\cong \matH^2$. Observe that this definition is independent of $p$.

From the computation of the dihedral angle
of the regular ideal geodesic $n$-simplex, together with the fact that geodesic simplices
of almost maximal volume are close in shape to regular ideal simplices, one deduces:

\begin{lemma}\label{maximal-angle}
Let $n\geqslant 4$. Then, there exists
$\vare_n>0$, depending only on $n$, such that the 
following condition holds:
if $\Delta\subseteq \mathno$ is a geodesic $n$-simplex such that
$\vol(\Delta)\geqslant (1-\vare_n)v_n$ and
$\alpha$ is the dihedral angle of $\Delta$
at any of its $(n-2)$-faces, then
$$
2< \frac{\pi}{\alpha} < 3\ .
$$
\end{lemma}

We refer the reader
to~\cite[Lemma 2.16]{FFM} for a proof.

\subsection*{Geometric straightening}\label{straight:sub}
Let us come back to the definition of straightening
for simplices in hyperbolic manifolds. 
Henceforth we denote by $M$ 
a  hyperbolic manifold with geodesic boundary. As usual, we also assume that $M$
is oriented.

The universal covering $\wdtM$ of $M$ is a convex subset of $\matH^n$ bounded 
by a countable family of disjoint geodesic hyperplanes (see \emph{e.g.}~\cite{Kojima1}).
If $\sigma\colon\Delta^k\to \wdtM$ is a singular $k$-simplex,
then we may define the simplex $\widetilde{\str}_k(\sigma)$ as follows: 
set $\widetilde{\str}_k(\sigma)(v) = \sigma(v)$ on every vertex $v$ of $\Delta^k$, and
extend using barycentric coordinates (see \cite[Chapter 11]{Ratcliffe}) or by an inductive cone construction
(which exploits the fact that any pair of points in $\wdtM$ is joined by a unique
geodesic, that continuously depends on its endpoints -- see \emph{e.g.}~\cite[Section 3.1]{FP} for full details). The image of $\widetilde{\str}_k(\sigma)$ is the geodesic simplex
spanned by the vertices of $\sigma$. Hence, we define a map $\widetilde{\str}_*:
C_*(\wdtM,\bb\wdtM;R)\rightarrow C_*(\wdtM,\bb\wdtM;R)$. Being $\pi_1(M)$-invariant, the map  $\widetilde{\str}_*$ induces a map
$$
\str_*\colon C_*(M,\bb M;R)\to C_*(M,\bb M;R)
$$
which is homotopic to the identity.
Simplices that lie in the image of $\str_*$ are called \emph{straight}. 

%Let us observe that $\str_*$ is in fact a straightening in the sense of \cite[Section 3]{BFP15}.

%(the 
%obvious inequality $\|\partial z\|_\Z\leq c(\partial P)$ 
%is an equality since by definition the pairings defining $P$ correspond to a \emph{maximal}
%set of canceling pairs of faces).

Recall that, by Lemma~\ref{integral:estimate:lemma}, the simplicial volume 
of a connected oriented $n$-manifold can be computed from \emph{integral} cycles.
Using the straightening procedure we show that such cycles
may be represented by $n$-pseudomanifolds with additional properties.

\begin{prop}\label{cyclehyp:prop}
Suppose that $M$ is a hyperbolic $n$-manifold with geodesic boundary $\bb M$.
Let $\vare>0$ be fixed. Then, there exists a relative integral cycle $z\in C_n (M,\partial M;\Z)$ 
with associated pseudomanifold $P$ such that the following conditions hold:
\begin{enumerate}
\item
$[z]=d\cdot [M,\bb M]_\Z$ in $H_n(M,\partial M;\Z)$ for some integer $d> 0$, and
$$ \frac{\|z\|_\Z}{d}\leq \| M,\bb M\| +\vare\ ;$$
\item
every singular simplex appearing in $z$ is straight;
\item
every simplex of $P$ has at most one $(n-1)$-dimensional boundary face;
\item
if $n=3$, then 
every simplex of 
$P$ without $2$-dimensional boundary  faces has at most
two edges contained in $\bb |P|$ and every simplex of $P$ has at most three edges in $\bb |P|$.
\end{enumerate}
\end{prop}
\begin{proof}
By Lemma \ref{integral:estimate:lemma} we may choose an integral cycle $z'$ satisfying condition (1), and set $z= \str_n(z')\in C_n(M,\bb M;\mathbb{Z})$. 
As usual,
we understand that no simplex appearing in $z$ is supported in $\bb M$ (otherwise,
we may just remove it from $z$ without modifying 
the class of $z$ in $C_n(M,\bb M;\mathbb{Z})$ and decreasing
$\|z\|_\mathbb{Z}$).
Point~(1) descends from the fact that
the straightening operator is norm nonincreasing and homotopic to the identity, while point (2) is obvious.

Let $\sigma$ be a straight $n$-simplex of $z$.
Let $\widetilde{\sigma}$ be a fixed lift of $\sigma$ to $\wdtM$. If there exists a component of $\bb\wdtM$ containing
 $m+1$ vertices of $\widetilde{\sigma}$, then  $\widetilde{\sigma}$ has an $m$-dimensional face supported on $\bb\wdtM$.
% there exists a simplex $\widetilde{\sigma}'\in S_k(\wdtM)$ having the same vertices
% of $\widetilde{\sigma}$ and an $m$-dimensional face supported on $\bb\wdtM$. This implies
%that $\widetilde{\sigma}=\strtil_k(\widetilde{\sigma}')$ also has an $m$-dimensional face supported in $\bb \wdtM$.
 In particular, the assumption that $\sigma$ is not supported on $\bb M$ implies that no  component of $\bb\widetilde{M}$ contains all the vertices of $\widetilde{\sigma}$. Hence,
\begin{enumerate}[(i)]
 \item If $\sigma$ has two faces on $\bb M$, then the vertices of $\widetilde{\sigma}$ 
are contained in the same connected component of $\bb \widetilde{M}$, a contradiction. 
\item  Suppose  that $\sigma$ has at least three edges on $\bb M$. If $n=3$, the 
union of the corresponding edges of $\widetilde{\sigma}$ is connected, so
at least three vertices
of $\widetilde{\sigma}$ lie on the same connected component of $\bb\wdtM$, and 
at least one $2$-face of $\sigma$ is 
supported on $\bb M$.

 If four edges of $\sigma$ lie on $\bb M$ and $n=3$, then as before, the union of the corresponding edges of $\widetilde{\sigma}$ is connected. 
But the vertices of these four edges of the $3$-simplex are all the vertices of the $3$-simplex, which all lie on the same connected component of $\bb\wdtM$, 
a contradiction.
\end{enumerate}
Now points (3) and (4) immediately descend from (i) and (ii).
\end{proof}

%Recall from Proposition~\ref{cycle:prop} 
%that, once $\vare>0$ is fixed, one can choose an integral cycle 
%choose an integral cycle
%$z\in C_3(M,\partial M;\Z)$ with associated
%gluing $P$ that realizes the simplicial volume of $M$, up to $\vare$.
%However, since $M$ is hyperbolic, we can improve Proposition~\ref{cycle:prop} as follows:
%\begin{prop}\label{cyclehyp:prop}
%Let $\vare>0$ be fixed. Then, there exists a relative integral cycle $z\in C_3 (M,\partial M;\Z)$ 
%with associated gluing $P$ such that the following conditions hold:
%\begin{enumerate}
%\item
%$[z]=d\cdot [M,\bb M]_\Z$ in $H_3(M,\partial M;\Z)$ for some integer $d\neq 0$, and
%$$ \frac{\|z\|_\Z}{|d|}\leq \| M,\bb M\| +\vare\ ;$$
%\item
%$P$ is connected; 
%\item
%every singular simplex appearing in $z$ is straight (so every simplex of the gluing
%$P$ has at most one boundary facet).
%\end{enumerate}
%\end{prop}
%\begin{proof}
% Since the straightening operator is norm non-increasing and is homotopic to the identity, %we may choose
%a real fundamental cycle which is supported on straight simplices
%and has norm bounded from above by $\|M,\bb M\|+\vare/2$. The proof of %Proposition~\ref{rational:prop} shows that such a cycle may be approximated by
%a rational cycle which is still supported on straight simplices. It is now easy to check
%that the proof of Proposition~\ref{cycle:prop} applies verbatim 
%and provides the desired integral cycle. 
%\end{proof}
\subsection*{Volume form}
Let $\sigma\colon \Delta^n\to M$ be a smooth $n$-simplex,
and let $\omega$ be the volume form of $M$. We set
$$
\vola (\sigma)=\int_{\Delta^n} \sigma^*(\omega)\ .
$$
%(since $\sigma$ is straight, it is smooth, so the above integral makes sense).
Since straight simplices are smooth, the map 
$$
C_n(M,\bb M;\R)\to \R,\qquad \sum_{i=1}^n a_i\sigma_i\mapsto \sum_{i=1}^n a_i\vola (\str_n(\sigma_i))
$$
is well-defined. This map
is a relative cocycle that represents the volume coclass on $M$ (see \emph{e.g.}~\cite[Section 4]{FP} for the details).
Therefore, if $z=\sum_{i=1}^h a_i\sigma_i\in C_n(M,\bb M;\Z)$ is an integral cycle
supported by straight simplices such that $[z]=d\cdot [M,\bb M]_\Z$, then
\begin{equation}\label{vol:eq}
\sum_{i=1}^h a_i\vola(\sigma_i)=d\cdot \vol(M)\ .
\end{equation}
Let us rewrite $z$ as follows:
$$
z=\sum_{i=1}^N \vare_i \sigma_i\ ,
$$
where $\vare_i=\pm 1$ for every $i=1,\ldots,N$. Note that we do not assume that $\sigma_i\neq \sigma_j$ for $i\neq j$. 
Let 
$P$ be the pseudomanifold associated to $z$, and recall that the simplices
$\Delta^n_1,\ldots,\Delta^n_N$ of $P$ are in bijection with the
$\sigma_i$'s. An identification of $\Delta^n_{i}$ with the standard $n$-simplex
is fixed for every $i=1,\ldots,N$, so that we may consider $\sigma_{i}$ as a map defined
on $\Delta^n_{i}$. We set
\begin{equation}\label{equ: algebraic volume}
\vola(\Delta^n_i)=\vare_i\vola(\sigma_i)\ ,
\end{equation}
and we say that $\Delta^n_i$ is \emph{positive} (resp.~\emph{degenerate}, \emph{negative})
if $\vola(\Delta^n_i)>0$ (resp.~$\vola(\Delta^n_i)=0$, $\vola(\Delta^n_i)<0$).
Equation~\eqref{vol:eq} may now be rewritten as follows:
\begin{equation}\label{vol2:eq}
\sum_{i=1}^N \vola(\Delta^n_i)=d\cdot \vol(M)\ .
\end{equation}
If $\widetilde{\sigma}_i$ is any lift of $\sigma_i$ to $\wdtM\subseteq \matH^n$,
then $\Delta^n_i$ is degenerate if and only if the image of $\widetilde{\sigma}_i$
is. Since $|\vola(\Delta^n_i)|$ is just the volume of the image of $\widetilde{\sigma}_i$, 
by Theorem~\ref{maximal:teo} we have
$$
|\vola(\Delta^n_i)|\leq v_n\ .
$$
If $\Delta^n_i$ is nondegenerate and $F$ is an $(n-2)$-face of
$\Delta^n_i$, then we define the angle of $\Delta^n_i$ at $F$ as the angle
of the image of $\widetilde{\sigma}_i$ at $\widetilde{\sigma}_i(F)$.

\begin{lemma}\label{or}
Let $F$ be an $(n-2)$-face of $\bb P$, and let $\Delta^n_{i_1},\ldots,\Delta^n_{i_k}$
be the simplices of $P$ that contain $F$ (taken with multiplicities).
For every $j=1,\ldots,k$ we also suppose that $\vola(\Delta^n_{i_j})>0$, so in particular
$\Delta^n_{i_j}$ is nondegenerate, and has a well-defined angle $\alpha_{i_j}$ at $F$. Then
$$
\sum_{j=1}^k \alpha_{i_j} =\pi\ .
$$
\end{lemma}
\begin{proof}
Up to choosing suitable lifts $\widetilde{\sigma}_{i_j}$
of the $\sigma_{i_j}$'s, we may glue 
the $\widetilde{\sigma}_{i_j}$'s in order to develop the union
of the $\Delta^n_{i_j}$'s
into $\wdtM\subseteq \matH^n$. Since the $(n-1)$-faces of $\bb P$
sharing $F$ are developed into two adjacent $(n-1)$-geodesic simplices
in $\bb\wdtM$, this implies at once that
a suitable algebraic sum of the $\alpha_{i_j}$'s is equal either to $0$ or to $\pi$.
In order to conclude it is sufficient to show that the condition
$\vola(\Delta^n_{i_j})>0$ implies that all the signs in this algebraic sum are positive
(this implies in particular that the sum is itself positive, whence equal to $\pi$).

To prove the last statement,
it is sufficient to check that, if
$\Delta^n_{i_{j_1}}$, $\Delta^n_{i_{j_2}}$ 
are adjacent in $P$ along
their common $(n-1)$-face $V$ and the lifts $\widetilde{\sigma}_{i_{j_1}}$, $\widetilde{\sigma}_{i_{j_2}}$ coincide on $V$, then
the images of $\widetilde{\sigma}_{i_{j_1}}$ and $\widetilde{\sigma}_{i_{j_2}}$
lie on different sides of $\widetilde{\sigma}_{i_{j_1}}(V)=
\widetilde{\sigma}_{i_{j_2}}(V)$. Let us set for simplicity $j_1=1$ and $j_2=2$,
and for $j=1,2$ let $\vare'_{i_j}=1$ if $V$ is the $k$-th face of
$\Delta^n_{i_j}$ and $k$ is even, and $\vare'_{i_j}=-1$ otherwise.
It is easily checked that the images of $\widetilde{\sigma}_{i_1}$ and $\widetilde{\sigma}_{i_2}$ lie on different sides of $\widetilde{\sigma}_{i_1}(V)=
\widetilde{\sigma}_{i_1}(V)$ if and only if
the quantities
$$
\vare'_{i_1}\vola(\sigma_{i_1})\, , \qquad
\vare'_{i_2}\vola(\sigma_{i_2})
$$
have opposite sign.
However, since $V$ corresponds to a canceling pair, we have
$\vare_{i_1}\vare_{i_1}'+\vare_{i_2}\vare_{i_2}'=0$, so the conclusion follows
from the positivity of $\vola(\Delta^n_{i_1})$ and $\vola(\Delta^n_{i_2})$.
\end{proof}

\section{Proof of Theorem~\ref{hyp:high:thm}}\label{proof-main:sec}
Throughout this section we suppose that $\dim M=n\geq 4$. The idea of the proof is as follows: Lemma~\ref{maximal-angle} implies that
no dihedral angle of a geodesic $n$-simplex of almost maximal volume can be a submultiple of
$\pi$. Together with Lemma~\ref{or}, this implies that any fundamental cycle $M$ must contain simplices whose support has small volume (that is, smaller than $(1-\varepsilon_n)v_n$). 
In fact, the weights of these simplices in any fundamental cycle may be bounded from below 
by the simplicial volume of the boundary of $M$, and this will finally yield the estimate needed in Theorem~\ref{hyp:high:thm}. Let us now provide the detailed computations.

Let $I=\{1,\ldots,N\}$ and let
$$
z=\sum_{i\in I} \vare_i \sigma_i 
$$
be an integral $n$-cycle satisfying the conditions of Proposition~\ref{cyclehyp:prop},
where $\vare_i=\pm 1$ for every $i\in I$. Let $P$ be a pseudomanifold associated to $z$, and let $\Delta^n_i$, $\vola(\Delta^n_i)$ 
be defined as in the previous section.

We choose $\vare_n$ as in Lemma~\ref{maximal-angle}, we say that the simplex $\Delta_i^n$ is \emph{small}
if and only if $\vola(\Delta_i^n)\leq (1-\vare_n)v_n$, and we
set 
$$I_{\mathrm{small}}=\{i\in I\, |\, \Delta_i^n\ \textrm{is\ small}\},\qquad
N_{\mathrm{small}}=\# I_{\mathrm{small}}
%I_\mathrm{small}=I\setminus I_b
\ .$$

\begin{lemma}\label{boundN}
 We have
$$
N_{\mathrm{small}}\geq\frac{d}{n+1}\|\bb M\|\ .
$$
\end{lemma}
\begin{proof}
We start by showing that every $(n-2)$-face of $\bb |P|$
is contained in at least one small $n$-simplex $\Delta^n_i$ of $P$, with  $i\in I_{\mathrm{small}}$,  corresponding to some $\sigma_i$. Indeed, let $F$ be an $(n-2)$-face of $\bb |P|$ and let $\Delta^n_{i_1},\ldots,\Delta^n_{i_k}$
be the $n$-simplices of $P$ containing $F$. Suppose
by contradiction that $\vola(\Delta^n_{i_j})\geq (1-\vare_n)v_n$ for every $j=1,\ldots,k$.
Let $\sigma_{i_j}$ be the
straight simplex corresponding to $\Delta^n_{i_j}$. Our assumptions imply that
the dihedral angle $\alpha_{i_j}$
of $\sigma_{i_j}(\Delta^n_{i_j})$ at $F$ is well-defined. Moreover, 
Lemma~\ref{or} gives
$$
\sum_{j=1}^k \alpha_{i_j}=\pi\ ,
$$
which contradicts 
Lemma~\ref{maximal-angle}.

Of course, a small simplex could have several $(n-2)$-faces in the boundary, but since an $n$-simplex has exactly $(n+1)n/2$ faces
of codimension two, we can bound the number of small simplices by the number of $(n-2)$-dimensional faces in $\bb |P|$,
$$ N_{\mathrm{small}} \geq \frac{2}{(n+1)n} \sharp \{ (n-2)\mathrm{-faces \ in \  }  \bb |P|\}.$$
An $(n-1)$-simplex has exactly $n$ faces of codimension one. 
Moreover, 
since $\bb P$ is an $(n-1)$-dimensional pseudomanifold without boundary,
% each $(n-1)$-simplex
%has $n$ codimension $1$ faces, and 
every $(n-2)$-face of $\bb |P|$ is shared by exactly two
$(n-1)$-simplices, so the number of $(n-2)$-faces of $\bb |P|$
is equal to $(n/2)c( \bb P)$, where $c(\bb P)$ is the number of $(n-1)$-simplices of $\bb P$. 
So inequality \eqref{eq:boundary}, i.e.
$$
c(\bb P)=\|\bb z\|_\mathbb{Z}\geq d\cdot \|\bb M\|\ ,
$$
concludes the proof of the lemma.
\end{proof}

To conclude the proof of Theorem~\ref{hyp:high:thm}, note that by equation~\eqref{vol2:eq} we have
%$$
\begin{align*}
d\cdot \vol(M)&=\sum_{i\in I}\vola(\Delta_i^n)=
\sum_{i\in I_\mathrm{small}} \vola(\Delta_i^n)+\sum_{i\in I\setminus I_\mathrm{small}} \vola(\Delta_i^n)\\
&\leq N_{\mathrm{small}}(1-\vare_n)v_n+(N-N_{\mathrm{small}})v_n = (N-N_{\mathrm{small}}\cdot \vare_n)v_n\ .
\end{align*}
Putting together this inequality with Lemma~\ref{boundN} 
and the inequality
 $N=\|z\|_\mathbb{Z}\leq d(\|M,\bb M\|+\vare)$ we get
$$
d\cdot \vol(M)\leq \left(d(\|M,\bb M\|+\vare)-\frac{d\cdot \vare_n}{n+1}\|\bb M\|\right)v_n \ .
$$
As $\varepsilon$ is arbitrary, after dividing each side of this inequality by $d\cdot \vol(M) $ and reordering, 
 we get
$$
\frac{\| M,\bb M\|}{\vol(M)}\geq \frac{1}{v_n}+\frac{\vare_n\cdot \|\bb M\|}{(n+1)\vol(M)}=
\frac{1}{v_n}+\frac{\vare_n\vol(\bb M)}{(n+1)v_{n-1}\vol(M)}\ ,
%v_n-\frac{\vare_n \cdot \|\bb M\|}{(n+1)\cdot \|M,\bb M\|}\geq 
%\frac{\vol(M)}{\|M,\bb M\|}\ .
$$
which finishes the proof of Theorem~\ref{hyp:high:thm}.

%Recall now that $\|\bb M\|=\vol(\bb M)/v_{n-1}$ and $\|M,\bb M\|>\vol(M)/v_n$. As a consequence, the left hand side of the last inequality is bounded above by
%$$
%v_n-\frac{\vare_n v_n \vol(\bb M)}{(n+1)v_{n-1}\vol(M)}\ .
%$$
%Therefore, we have finally proved that 
%$$
%\frac{\vol(M)}{\|M,\bb M\|}\leq v_n-\frac{\vare_n v_n}{(n+1)v_{n-1}}\cdot \frac{\vol(\bb %M)}{\vol(M)}\ ,
%$$
%and this concludes the proof of Theorem~\ref{hyp:high:thm}.

\section{Proof of Theorem~\ref{hyperbolic:thm}}\label{proof-hyp:sec}%qui

%\subsection*{Some estimates on volumes of hyperbolic $3$-simplices}\label{Catalan:sub}
In order to provide lower bounds on the simplicial volume of hyperbolic
$3$-manifolds with geodesic boundary we first need to analyze some properties
of volumes of hyperbolic $3$-simplices. An essential tool for computing
such volumes is the \emph{Lobachevsky function} $L\colon \R\to \R$
defined by the formula
$$
L(\theta)=-\int_0^\theta \log |2\sin u|\, du\ .
$$
In a nondegenerate ideal $3$-simplex, opposite sides subtend
isometric angles, the sum of the angles of any triple of edges
sharing a vertex is equal to $\pi$ and the simplex
is determined up to isometry by its dihedral angles.
The following result is proved by Milnor in~\cite[Chapter 7]{Thurston}, and
plays a fundamental role in the computation of volumes
of hyperbolic $3$-simplices. %or the description of several properties
%of the Lobachevsky function).

\begin{prop}\label{milnor:prop}
 Let $\Delta$ be a nondegenerate ideal simplex with angles $\alpha,\beta,\gamma$.
Then
$$
\vol(\Delta) = L(\alpha)+L(\beta)+L(\gamma)\ .
$$
Moreover, 
$$
\vol(\Delta)\leq 3L(\pi/3)=v_3\approx 1.014942\ ,
$$
where the equality holds if and only if $\alpha=\beta=\gamma=\pi/3$ (\emph{i.e.}~$\Delta$ is regular).
\end{prop}

We say that a nondegenerate geodesic simplex with nonideal vertices
$\Delta$ is:
\begin{itemize}
 \item 
 \emph{1-obtuse} if it has at least
one nonacute dihedral angle, 
\item 
\emph{2-obtuse} if there exist two edges of $\Delta$
which share a vertex and subtend nonacute dihedral angles,
\item
\emph{3-obtuse} if there exists a face $F$ of $\Delta$
such that each edge of $F$ subtends a nonacute dihedral angle.
\end{itemize}

\begin{lemma}\label{3-obtuse}
 There do not exist $3$-obtuse geodesic simplices. Moreover,
if $\Delta$ is a $2$-obtuse geodesic simplex, then
$$
\vol(\Delta)\leq  \frac{v_3}{2}\ .
$$
\end{lemma}

\begin{proof} Let $F$ be a face of a nondegenerate geodesic simplex $\Delta$. Let $H$ be the geodesic plane containing $F$ and let $\pi\colon\matH^3\to H$ denote the nearest point projection. 
Let $v\in \mathno$ be the vertex of $\Delta$ not contained in $H$. 

For every edge $e$ of $F$, the geodesic line containing $e$ divides $H$ into two regions. Note that the angle at $e$ is acute if and only if the projection  $\pi(v)$ 
of the last vertex point belongs to the region containing $F$. Consider the three geodesic lines containing the three edges of $F$. Since no point in $H$ 
can simultaneously be contained in the region of $H$ bounded by each of these geodesics and not containing $F$, it follows that $\Delta$ cannot be $3$-obtuse. 

Suppose now that two of the edges of $F$ subtend nonacute dihedral angles and consider the four regions of $H$ delimited by the two corresponding geodesics. 
Denote by $v_0$ the vertex of $F$ given as the intersection of these two geodesics. Note that  $\pi(v)$ belongs to the region opposite to the region containing $F$. 
Denote by $r$ the reflection along $H$. Set $v'=r(v)$ and $\Delta'=r(\Delta)$. The convex hull of $\Delta$ and $\Delta'$ is equal to the convex hull of $F$,$v$ and $v'$. 
Let $\widehat{\Delta}$ be the geodesic simplex with vertices $v,v'$ and the two vertices of $F$ opposite to $v_0$. 
Since $v_0$ belongs to $\widehat{\Delta}$ (see Figure~\ref{delta:fig}) it follows that $\Delta \cup \Delta'\subset \widehat{\Delta}$ and hence
$$
v_3\geq \vol(\widehat{\Delta})\geq \vol(\Delta\cup \Delta' )=2\vol(\Delta)\ ,
$$
which finishes the proof of the lemma.
\end{proof}
\begin{figure}
\begin{center}
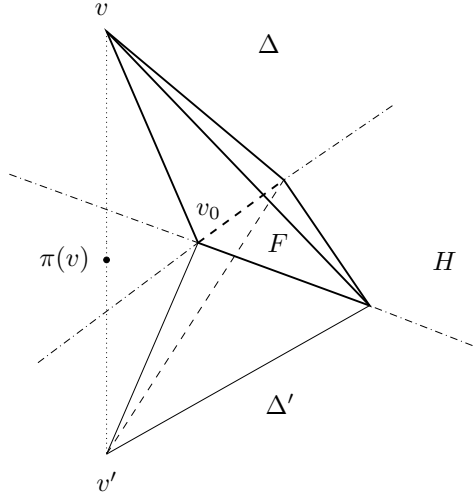
\caption{A $2$-obtuse simplex $\Delta$ and its mirror copy $\Delta'$.\protect{\label{delta:fig}}}
\end{center}
\end{figure}

Recall that 
$$
G=\frac{1}{1^2}-\frac{1}{3^2}+\frac{1}{5^2}-\frac{1}{7^2}+\ldots\,  \approx 0.915965
$$ 
is Catalan's constant.

\begin{lemma}\label{1-obtuse}
If $\Delta$ is a $1$-obtuse geodesic simplex, then
$$
\vol(\Delta)\leq  G\ .
$$
%where $L$ is the Lobachevsky function (which vanishes on $\pi/2$) and $G$ is Catalan's constant. The maximum
%is attained at the ideal simplex with angles $\pi/2,\pi/4,\pi/4$.
\end{lemma}
\begin{proof}
Suppose first that $\Delta$ is a $1$-obtuse ideal geodesic simplex. Let $\alpha,\beta,\gamma$ be its three dihedral angles with $\alpha \geq \pi/2$ and $\beta+\gamma=\pi-\alpha$. 
Using Proposition~\ref{milnor:prop}, one can readily show that when $\alpha\geq \pi/2$ is fixed, the maximum volume $\vol(\Delta)=L(\alpha)+L(\beta)+L(\gamma)$ is attained at $\beta=\gamma=(\pi-\alpha)/2$. 
Another easy computation based on Proposition~\ref{milnor:prop} implies
that, under the assumption that $\alpha\geq \pi/2$, the quantity
$L(\alpha)+2L((\pi-\alpha)/2)$ attains its maximum at $\alpha=\pi/2$. Therefore,
we may conclude that
$$
\vol(\Delta)\leq L(\pi/2)+2L(\pi/4)= G \ ,
$$
where the last equality is proved in~\cite[Chapter 7]{Thurston}.

Let now $\Delta$ be a $1$-obtuse nonideal geodesic simplex. The lemma will follow once we exhibit a $1$-obtuse ideal geodesic simplex $\overline{\Delta}$ with $\Delta\subset \overline{\Delta}$.
 Let $v_1,v_2$ be the vertices on the edge $e$ subtending the nonacute angle. Two of the vertices of $\overline{\Delta}$ will be the two endpoints $w_1,w_2$ of the geodesic 
through $v_1,v_2$. Let $v,v'$ be the two remaining vertices of $\Delta$ and denote by $F,F'$ the two faces of $\Delta$ opposite to $v'$ and $v$ respectively. 
Let $w$, respectively $w'$, be vertices on the boundary of the hyperplane containing $F$, resp.~$F'$, and such that the convex hull of $v_1,v_2,w$, resp.~$v_1,v_2,w'$, 
contains $v$, resp.~$v'$. (For example, pick $w$, resp.~$w'$, as the endpoint of the geodesic through $v_1$ and $v$, resp.~$v'$.) 
Let $\overline{\Delta}$ be the ideal geodesic simplex with vertices $w_1,w_2,w,w'$. As it contains all the vertices of $\Delta$, 
the simplex $\Delta$ is indeed contained in $\overline{\Delta}$. Furthermore, it is still $1$-obtuse as its dihedral angle on the edge with endpoints $w_1,w_2$ 
is equal to the dihedral angle of $\Delta$ at the edge with endpoints $v_1,v_2$.  
\end{proof}

\subsection*{Proof of Theorem~\ref{hyperbolic:thm}}\label{hyp:last:sub}
%\begin{proof}[Proof of Theorem~\ref{hyperbolic:thm}]\label{hyp:last:sub}
Let $z$ be the integral cycle provided by Proposition~\ref{cyclehyp:prop},
let $P$ be the associated pseudomanifold, and 
let $\Delta^3_1,\ldots,\Delta^3_N$ be the simplices of $P$.
In equation~(\ref{equ: algebraic volume}) a well-defined algebraic volume 
$\vola(\Delta^3_i)$ is associated to 
every $\Delta^3_i$, in such a way that the equality
\begin{equation}\label{voltot:eq}
d\cdot \vol(M)= \sum_{i=1}^N \vola(\Delta^3_i)
\end{equation}
holds. We also say that $\Delta^n_{i}$ is $1$-, $2$- or $3$-obtuse
if the corresponding geodesic simplex in $\matH^n$ is
(by Lemma~\ref{3-obtuse}, the last possibility
cannot hold in fact).

Let $\Omega_i$, $i=0,\ldots,4$, be the set of simplices
of $P$ having exactly $i$ boundary  $2$-faces. 
We denote by $t_i$ the number of elements
of $\Omega_i$. By Proposition~\ref{cyclehyp:prop}
we have $\Omega_2=\Omega_3=\Omega_4=\emptyset$, so that
$$
t_2=t_3=t_4=0,\qquad \| z\|_\Z=t_0+t_1=N\ .
$$
We denote by $t_{1,n}$ the number
of nonpositive simplices in $\Omega_1$ (\emph{i.e.} simplices with nonpositive volume), and by $t_{1,1}$ (resp.~$t_{1,2}$)
the number of $1$-obtuse (resp.~$2$-obtuse) positive simplices in $\Omega_1$.
%Recall from Section~\ref{aspherical3:sec} that 
%an edge of the $2$-dimensional pseudomanifold $\bb P$ 
%is nice if it is contained in one element of $\Omega_0$, and bad otherwise, and
%an edge of  $\bb |P|$ is nice if it is the image of at least one nice edge in $\bb P$, and
%bad otherwise.
Moreover, we say that an edge $e$ of the $2$-dimensional pseudomanifold $\bb P$
is \emph{nice} if $e$ is the edge of a
simplex in $\Omega_0$.
We also say that an edge of $\bb P$ is \emph{bad} if it is not nice.
An edge of the topological realization $\bb |P|$ is \emph{nice} if it is the image of at least one nice edge in $\bb P$.
An edge in $\bb |P|$ is \emph{bad} if it is not nice.
Notice that bad edges in $\bb|P|$ are \emph{not} the image of bad edges in $\bb P$,
since a nice edge in $\bb |P|$ will be the image of a certain number of
nice edges in $\bb P$ and necessarily two bad edges of $\bb P$ corresponding to the
two (possibly nondistinct) tetrahedra having as 2-faces the 2-faces in $\bb |P|$ containing the
original nice edge.
We denote by $E_\mathrm{bad}$ (resp.~$E_\mathrm{nice}$) the number
of bad (resp.~nice) edges of $\bb |P|$. 
%an edge $e$ of the $2$-dimensional pseudomanifold $\bb P$ is \emph{nice} if it is contained in at least one element of $\Omega_0$,
%and \emph{bad} otherwise.
%We denote by $E_\mathrm{bad}$ (resp.~$E_\mathrm{nice}$) the number
%of bad (resp.~nice) edges of $\bb P$. 
%
\begin{lemma}\label{E:obtuse:lemma}
 We have
$$
3t_{1,n}+2t_{1,2}+t_{1,1}\geq E_\mathrm{bad}\ .
$$
\end{lemma}
\begin{proof}
Let $e$ be a bad edge of $\bb |P|$, let $T_{i_1}$, $T_{i_2}$ be the triangles of $\bb P$ adjacent to $e$, and let  $\Delta^3_{i_1}$ (resp.~$\Delta^3_{i_2}$) be the simplex of $P$ containing $T_{i_1}$ (resp.~$T_{i_2}$). 
It is easy to show that if $F_{i_1}$ (resp.~$F_{i_2}$) is the $2$-face of $\Delta^3_{i_1}$ (resp.~$\Delta^3_{i_2}$) such that
$e=F_{i_1}\cap T_{i_1}=F_{i_2}\cap T_{i_2}$, then $F_{i_1}, F_{i_2}$ are glued to each other in $|P|$
(see \cite[Lemma 4.4]{BFP15}). Also suppose that $\Delta^3_{i_1}$ and $\Delta^3_{i_2}$ are both positive
and let $\alpha_{i_j}$ be the dihedral angle of $\Delta^3_{i_j}$ at $e$. By Lemma~\ref{or}
we have $\alpha_{i_1}+\alpha_{i_2}=\pi$.
As a consequence, either $\Delta^3_{i_1}$ or $\Delta^3_{i_2}$ (or both in case $\alpha_{i_1}=\alpha_{i_2}=\pi/2$) has a nonacute angle
along $e$. We have thus shown that at every bad edge of $\bb |P|$ there is (at least)
one incident
simplex that either is nonpositive or 
has a nonacute angle at an edge of its
boundary face. Since we know that no simplex of $P$ can be $3$-obtuse, the conclusion follows from an obvious double counting argument.
\end{proof}

\begin{prop}\label{estimate1:prop}
We have 
$$
\| M,\bb M\|+\vare \geq \frac{\vol(M)}{v_3}+\left(1-\frac{G}{v_3}\right)\frac{E_\mathrm{bad}}{d}\ .
$$ 
\end{prop}
\begin{proof}
Since $v_3\geq 3(v_3-G)$ and $v_3/2\geq 2(v_3-G)$, 
by equation~\eqref{voltot:eq} and Lemmas~\ref{3-obtuse}, \ref{1-obtuse} and
\ref{E:obtuse:lemma} we have
\begin{align*}
 d\cdot \vol(M)&  \leq (t_0+t_1-t_{1,1}-t_{1,2}-t_{1,n})v_3+Gt_{1,1}+t_{1,2}\frac{v_3}{2}\\
& =(t_0+t_1)v_3-(v_3-G)t_{1,1}-t_{1,2}\frac{v_3}{2}-t_{1,n}v_3\\ 
& \leq (t_0+t_1)v_3-(v_3-G)(t_{1,1}+2t_{1,2}+3t_{1,n})\\ 
& \leq (t_0+t_1)v_3-(v_3-G)E_\mathrm{bad}\ .
\end{align*}
Now the conclusion follows from the inequality
$t_0+t_1=\|z\|_\Z\leq d(\| M,\bb M\|+\vare)$
(see Proposition~\ref{cyclehyp:prop}).
\end{proof}

\begin{prop}\label{estimate2:prop}
 We have
$$
\| M,\bb M\|+\vare \geq \frac{7}{4} \|\bb M\| -\frac{E_\mathrm{bad}}{2d}\ .
$$
\end{prop}
\begin{proof}
%Just as in the proof of Theorem~\ref{main:thm} and 
Recall that $\bb |P|$ is the union of a finite number of closed orientable surfaces, and that the pseudomanifold $P$ comes equipped
with a degree $d$ map $f\colon (|P|,\partial |P|)\rightarrow (M,\partial M)$.
Therefore, we have
that $\| \partial |P|\|\geq d\cdot \|\partial M\|$.
Since $\bb |P|$ is decomposed into $t_1$ triangles, this implies
that 
\begin{equation}\label{t_1:eq}
t_1\geq d\cdot \|\bb M\|\ .
\end{equation}

For the number $E_\mathrm{nice}$ of nice
edges  of $\bb |P|$ 
we have the obvious equality
$E_\mathrm{nice}= (3/2)t_1-E_\mathrm{bad}$. 
By definition, 
every nice edge is contained in a simplex in $\Omega_0$,
and by point (4) of Proposition~\ref{cyclehyp:prop} any such simplex has at most
$2$ edges on $\bb |P|$, so
\begin{equation}\label{t_0:eq}
 t_0\geq \frac{E_\mathrm{nice}}{2}=\frac{3}{4}t_1-\frac{E_\mathrm{bad}}{2}\ .
\end{equation}
Together with Proposition~\ref{cyclehyp:prop}, Inequalities~\eqref{t_1:eq}
and~\eqref{t_0:eq} imply that
$$
d(\|M,\bb M\|+\vare)\geq t_0+t_1\geq \frac{7}{4}t_1-\frac{E_\mathrm{bad}}{2}\geq \frac{7d}{4}\cdot \|\bb M\|-\frac{E_\mathrm{bad}}{2}\ ,
$$
which finishes the proof of the proposition.
\end{proof}

We are now ready to prove Theorem~\ref{hyperbolic:thm}. In fact, if we set $k_0=E_\mathrm{bad}/(2d)$,
then putting together
Propositions~\ref{estimate1:prop} and~\ref{estimate2:prop} we get
$$
\|M,\bb M\|+\vare \geq \max\left\{\frac{\vol(M)}{v_3}+2\left(1-\frac{G}{v_3}\right)k_0,\,
\frac{7}{4}\cdot \|\bb M\|-k_0\right\}\ ,
$$
whence
$$
\|M,\bb M\|+\vare \geq \min_{k\geq 0} \max\left\{\frac{\vol(M)}{v_3}+2\left(1-\frac{G}{v_3}\right)k,\,
\frac{7}{4}\cdot \|\bb M\|-k\right\}\ .
$$
If $(7/4)\|\bb M\|\leq \vol(M)/v_3$, then the statement of Theorem~\ref{hyperbolic:thm}
is an obvious consequence of Jungreis' inequality~\eqref{jung:eq}. Otherwise,
the right-hand side of the inequality above is equal to
$$
\frac{\vol(M)}{v_3}+\frac{v_3-G}{2(3v_3-2G)}\left(7\|\bb M\|-4\frac{\vol(M)}{v_3}\right)
%\frac{\vol(M)}{v_3}+ \frac{v_3}{3v_3-2G}\left(\frac{7}{4}\|\bb M\|- \frac{\vol(M)}{v_3}\right)
%=\frac{7\|\bb M\| (v_3-G)+2\vol(M)}{2(3v_3-2G)}\ ,
$$
which finishes the proof of Theorem~\ref{hyperbolic:thm} since $\vare$ is arbitrary.%\end{proof}

\section{Small hyperbolic manifolds with geodesic boundary}\label{last:sec}
We start by recalling some results from~\cite{FriMaPe} and~\cite{KM}.
An ideal triangulation of a $3$-manifold $M$
is a homeomorphism between $M$ and $|P|\setminus V(|P|)$, where $P$
is a $3$-pseudomanifold and $V(|P|)$ is a regular open neighbourhood of the 
vertices of $|P|$. In other words, it is a realization of $M$ as the space obtained
by gluing some topological \emph{truncated tetrahedra}, \emph{i.e.} tetrahedra with
neighbourhoods of the vertices removed (see Figure~\ref{truncated:fig}).

\begin{figure}
\begin{center}
\includegraphics{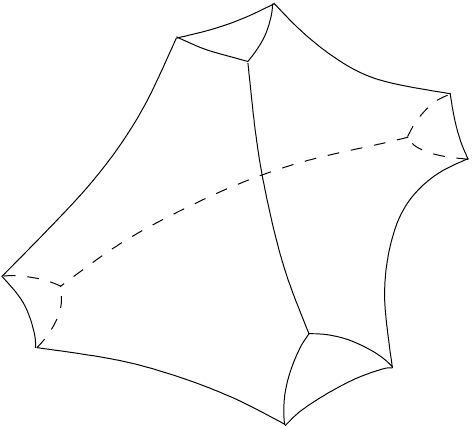}
\caption{A truncated tetrahedron.\protect\label{truncated:fig}}
\end{center}
\end{figure}

As in the introduction, let $\calM_g$, $g\geq 2$, be the class of
3-manifolds with boundary $M$ that admit an ideal triangulation by $g$ tetrahedra
and have Euler characteristic $\chi(M)=1-g$ (so $\chi(\bb M)=2-2g$).
We also denote by $\overline\calM_g$ the set
of hyperbolic $3$-manifolds $M$ with connected geodesic boundary such that 
$\chi(\bb M)=2-2g$. 
For $g\geq 2$, let $\Delta_g\subseteq \matH^3$ be the regular truncated tetrahedron
of dihedral angle $\pi/(3g)$ (see \emph{e.g.}~\cite{Kojima1, FriPe} for
the definition of hyperbolic truncated tetrahedron). It is proved in~\cite{Kojima} that
\begin{align}
\vol(\Delta_g)&=8L\left(\frac{\pi}{4}\right)-3\int_0^{\frac{\pi}{3g}}
{\rm arccosh}\, \left(\frac{\cos t}{2\cos t-1}\right)\, {\rm d}t \notag\\
\label{voldg}\\
& =4G-3\int_0^{\frac{\pi}{3g}}
{\rm arccosh}\, \left(\frac{\cos t}{2\cos t-1}\right)\, {\rm d}t\ . \notag
\end{align}

The following result lists some known properties of manifolds belonging to $\calM_g$.
The last point implies that $\calM_g$ coincides with the set of the elements
of $\overline\calM_g$ of smallest volume.

\begin{prop}[\cite{FriMaPe,KM}]\label{Mg:prop}
Let $g\geq 2$. Then:
\begin{enumerate}
\item
the set $\calM_g$ is nonempty;
\item 
every element of $\calM_g$ admits a hyperbolic structure with geodesic boundary
(which is unique up to isometry by Mostow Rigidity Theorem);
\item
the boundary of every element of $\calM_g$ is connected, so $\calM_g\subseteq \overline\calM_g$; 
\item
if $M\in\calM_g$, then $M$ decomposes into the union of $g$ copies of $\Delta_g$, so
in particular $\vol(M)=g\vol(\Delta_g)$;
\item
if $M\in\overline\calM_g$,
then $\vol(M)\geq g\vol(\Delta_g)$.
\end{enumerate}
\end{prop}

Items (2) and (3) and (4) are proved in~\cite{FriMaPe}, 
items~(1) and (5) in~\cite{KM}.

\begin{prop}
 Fix $g\geq 2$. Then all the elements of $\calM_g$ share the same simplicial volume.
\end{prop}
\begin{proof}
 Take $M_1,M_2\in\calM_g$, and let us consider the universal
coverings $\wdtM_1$ and $\widetilde{M}_2$. Both $\wdtM_1$ and $\widetilde{M}_2$
are obtained as the union in $\matH^3$ of a countable family of 
copies of $\Delta_g$, which are adjacent along their hexagonal faces, and this easily implies that $\wdtM_1$ and $\widetilde{M}_2$ are isometric
to each other. Since the isometry group of $\wdtM_1$ is discrete, this fact can be used to
show that $M_1$ and $M_2$ are commensurable, \emph{i.e.}~there exists a hyperbolic $3$-manifold
with geodesic boundary $M'$ that is the total space of finite coverings
$p_1\colon M'\to M_1$ and $p_{2}\colon M'\to M_2$ (see \cite[Lemma 2.4]{Fricomm}).
Since the Riemannian volume and the simplicial volume are multiplicative with respect to finite coverings, 
this implies in turn that $\|M_1,\bb M_1\|/\vol(M_1)=\|M_2,\bb M_2\|/\vol(M_2)$, which finishes the proof since $\vol(M_1)=\vol(M_2)$.
\end{proof}

Let us prove Corollary~\ref{minimal:cor} and see that for $M\in \calM_2 \cup \calM_3 \cup \calM_4$, the bounds provided by the corollary are indeed sharper 
than Jungreis' inequality~\eqref{jung:eq}
and  inequality~\eqref{54eq}. 
\begin{itemize}
\item
If $M\in\overline\calM_2$, then $\|\bb M\|=4$ and
$\vol(M)\geq 2\vol(\Delta_2)\approx 6.452$. Applying 
Theorem~\ref{hyperbolic:thm} we get
$$
\| M,\bb M\|\geq 6.461\approx 1.615 \cdot \|\bb M\| .
$$
Also observe that, if $M\in\calM_2$, then
$
{\vol(M)}/{v_3}\approx 6.357\ .
$
\item
If $M\in\overline\calM_3$, then $\|\bb M\|=8$ 
and $\vol(M)\geq 3\vol(\Delta_3)\approx 10.429$.
Applying 
Theorem~\ref{hyperbolic:thm} we get
$$
\| M,\bb M\|\geq 10.882\approx 1.360 \cdot \|\bb M\| .
$$
Also observe that, if $M\in\calM_3$, then
$
{\vol(M)}/{v_3}\approx 10.274\ .
$
\item
If $M\in\overline\calM_4$, then $\|\bb M\|=12$ and 
$\vol(M)\geq 4\vol(\Delta_4)\approx 14.238$.
Applying 
Theorem~\ref{hyperbolic:thm} we get
$$
\| M,\bb M\|\geq 15.165\approx 1.264 \cdot \|\bb M\| .
$$
If $M\in\calM_4$, then
$
{\vol(M)}/{v_3}\approx 14.097\ .
$
\end{itemize}

Finally, let us take $M\in\calM_g$, $g\geq 5$ and show that the bound 
$$\|M,\bb M\|\geq \frac{5}{4}\|\partial M\|$$ proved in~\cite{BFP15} is sharper than the ones given by 
inequality~\eqref{jung:eq} and  Theorem~\ref{hyperbolic:thm}. Note that it is sufficient to show that
%if $M\in\calM_g$, $g\geq 5$, then
$$
\frac{5}{4}\|\bb M\|>\frac{7\|\bb M\| (v_3-G)+2\vol(M)}{2(3v_3-2G)}>\frac{\vol(M)}{v_3}\ .
$$
Using that $\|\bb M\|=4(g-1)$ and $\vol(M)=g\vol(\Delta_g)$, 
%that is
%$$
%5(g-1)>\frac{28(g-1)(v_3-G)+2g\vol(\Delta_g)}{2(3v_3-2G)}>\frac{g\vol(\Delta_g)}{v_3}\ .
%$$
after some straightforward algebraic manipulations the first inequality and the second inequality may be rewritten respectively as follows:
$$
\left(1-\frac{1}{g}\right)(v_3+4G)>\vol(\Delta_g)\ ,\qquad
7\left(1-\frac{1}{g}\right)v_3>\vol(\Delta_g)\ .
$$
We know from equation~\eqref{voldg} that $\vol(\Delta_g)<4G$. Therefore, for every $g\geq 5$ we have
$$
\left(1-\frac{1}{g}\right)(v_3+4G)\geq \frac{4}{5}(v_3+4G)>4G>\vol(\Delta_g)
$$
and
$$
7\left(1-\frac{1}{g}\right)v_3\geq \frac{28}{5}v_3>4G>\vol(\Delta_g)\ ,
$$
whence the conclusion.

\bibliographystyle{amsalpha}
\bibliography{biblio}

\end{document}